\documentclass[a4paper,10pt]{article}
\usepackage[T2A]{fontenc}
\usepackage[utf8]{inputenc}

\usepackage[margin=2.5cm]{geometry}

\usepackage{amsmath, amssymb, amsthm}
\usepackage{hyperref}
\usepackage{color}

\def\F{{\mathbb F}}

\DeclareMathOperator{\Tr}{Tr}

\DeclareMathOperator{\ex}{ex}

\newtheorem{theorem}{Theorem}[section]

\newtheorem{corollary}[theorem]{Corollary}

\newtheorem{lemma}[theorem]{Lemma}

\begin{document}

\title{Combinatorial Nullstellensatz and Tur\'an numbers of complete $r$-partite $r$-uniform hypergraphs}
\author{Alexey Gordeev}
\date{}

\maketitle

\begin{abstract}
In this note we describe how Laso{\'n}'s generalization of Alon's Combinatorial Nullstellensatz gives a framework for constructing lower bounds on the Tur\'an number $\ex(n, K^{(r)}_{s_1,\dots,s_r})$ of the complete $r$-partite $r$-uniform hypergraph $K^{(r)}_{s_1,\dots,s_r}$.
To illustrate the potential of this method, we give a short and simple explicit construction for the Erd{\H o}s box problem, showing that $\ex(n, K^{(r)}_{2,\dots,2}) = \Omega(n^{r - 1/r})$, which asymptotically matches best known bounds when $r \leq 4$.
\end{abstract}

\section{Introduction}

\subsection{Tur\'an numbers of complete \texorpdfstring{$r$}{r}-partite \texorpdfstring{$r$}{r}-uniform hypergraphs}

A \textit{hypergraph} $H = (V, E)$ consists of a set of \textit{vertices} $V$ and a set of \textit{edges} $E$, each edge being some subset of $V$.
A hypergraph is \textit{$r$-uniform} if each edge in it contains exactly $r$ vertices.
An $r$-uniform hypergraph is \textit{$r$-partite} if its set of vertices can be represented as a disjoint union of $r$ parts with every edge containing one vertex from each     part.
The \textit{complete} $r$-partite $r$-uniform hypergraph with parts of sizes $s_1, \dots, s_r$ contains all $s_1 \cdots s_r$ possible edges and is denoted by $K^{(r)}_{s_1, \dots, s_r}$.

Let $H$ be an $r$-uniform hypergraph.
The \textit{Tur\'an number} $\ex(n, H)$ is the maximum number of edges in an $r$-uniform hypergraph on $n$ vertices containing no copies of $H$.
A classical result of Erd{\H o}s~\cite{erdosExtremalProblemsGraphs1964} implies that for $s_1 \leq \dots \leq s_r$,

\begin{equation}\label{eq:Erdos}
\ex(n, K^{(r)}_{s_1, \dots, s_r}) = O\left( n^{r - \frac{1}{s_1 \cdots s_{r - 1}}} \right).
\end{equation}

In~\cite{mubayiExactResultsNew2002}, Mubayi conjectured that bound~\eqref{eq:Erdos} is asymptotically tight.
Recently, Pohoata and Zakharov~\cite{pohoataNormHypergraphs2021} showed that this is true whenever $s_1, \dots, s_r \geq 2$ and $s_r \geq ((r - 1)(s_1 \cdots s_{r - 1} - 1))! + 1$, extending earlier results of Alon, Koll{\'a}r, R{\'o}nyai and Szab{\'o}~\cite{kollarNormgraphsBipartiteTuran1996,alonNormGraphsVariationsApplications1999} and Ma, Yuan and Zhang~\cite{maExtremalResultsComplete2018}.

Nevertheless, the conjecture remains open even in a special case $\ex(n, K^{(r)}_{2,\dots, 2})$, which is often referred to as \textit{the Erd{\H o}s box problem}.
The best known lower bound is due to Conlon, Pohoata and Zakharov~\cite{conlonRandomMultilinearMaps2021}, who showed that for any $r \geq 2$,

\begin{equation}\label{eq:CPZ}
\ex(n, K^{(r)}_{2,\dots, 2}) = \Omega\left( n^{r - \lceil \frac{2^r - 1}{r} \rceil^{-1}} \right).    
\end{equation}

\subsection{Generalized Combinatorial Nullstellensatz}

Let $\F$ be an arbitrary field, and let $f \in \F[x_1,\dots,x_r]$ be a polynomial in $r$ variables.
A monomial $x_1^{d_1} \cdots x_r^{d_r}$ is a \textit{monomial of a polynomial $f$} if the coefficient of $x_1^{d_1} \cdots x_r^{d_r}$ in $f$ is non-zero.
Recall the famous Combinatorial Nullstellensatz by Alon (see Theorem 1.2 in~\cite{alonCombinatorialNullstellensatz1999}).

\begin{theorem}[Alon, 1999]\label{thm:CN}
Let $x_1^{d_1} \cdots x_r^{d_r}$ be a monomial of $f$, and let $\deg f \leq d_1 + \dots + d_r$.
Then for any subsets $A_1, \dots, A_r$ of $\F$ with sizes $|A_i| \geq d_i + 1$, $f$ does not vanish on $A_1\times\dots\times A_r$, i.e. $f(a_1,\dots,a_r) \neq 0$ for some $a_i \in A_i$.
\end{theorem}

A monomial $x_1^{d_1} \cdots x_r^{d_r}$ of $f$ is \textit{maximal} if it does not divide any other monomial of $f$.
Laso{\'n} showed the following generalization of Combinatorial Nullstellensatz (see Theorem 2 in \cite{lasonGeneralizationCombinatorialNullstellensatz2010}).
It should be mentioned that an even stronger theorem was proved by Schauz in 2008 (see Theorem 3.2(ii) in~\cite{schauzAlgebraicallySolvableProblems2008}).

\begin{theorem}[Laso{\'n}, 2010]\label{thm:Lason}
Let $x_1^{d_1} \cdots x_r^{d_r}$ be a maximal monomial of $f$.
Then for any subsets $A_1, \dots, A_r$ of $\F$ with sizes $|A_i| \geq d_i + 1$, $f$ does not vanish on $A_1\times\dots\times A_r$, i.e. $f(a_1,\dots,a_r) \neq 0$ for some $a_i \in A_i$.
\end{theorem}

Notably, in most applications of Combinatorial Nullstellensatz the condition $\deg f \leq d_1 + \dots + d_r$ from Theorem~\ref{thm:CN} turns out to be sufficient and thus the more general Theorem~\ref{thm:Lason} is not needed.
Below we give a rare example of an application in which the full power of Theorem~\ref{thm:Lason} is essential.

\section{The framework}\label{sec:framework}

For subsets $B_1, \dots, B_r$ of a field $\F$ denote the set of zeros of a polynomial $f \in \F[x_1, \dots, x_r]$ on $B_1 \times \dots \times B_r$ as
\[
Z(f; B_1,\dots, B_r) := \{ (a_1, \dots, a_r) \in B_1 \times \dots \times B_r \ |\ f(a_1, \dots, a_r) = 0 \}.
\]
In the case $B_1 = \dots = B_r = B$ we will write $Z(f; B, r)$ instead of $Z(f; B_1, \dots, B_r)$.

The set $Z(f; B_1,\dots, B_r)$ can be viewed as the set of edges of an $r$-partite $r$-uniform hypergraph $H(f; B_1, \dots, B_r)$ with parts $B_1, \dots, B_r$.
Our key observation is the following lemma which immediately follows from Theorem \ref{thm:Lason}.

\begin{lemma}\label{lm:CN-Turan}
Let $x_1^{d_1} \cdots x_r^{d_r}$ be a maximal monomial of $f$.
Then for any subsets $B_1, \dots, B_r$ of $\F$ the hypergraph $H(f; B_1,\dots, B_r)$ is free of copies of $K^{(r)}_{d_1 + 1, \dots, d_r + 1}$.
\end{lemma}

This lemma gives us a new tool for constructing lower bounds on $\ex(n, K^{(r)}_{s_1, \dots, s_r})$.
In Section \ref{sec:Turan} we give a simple example of such construction for $\ex(n, K^{(r)}_{2, \dots, 2})$ which asymptotically matches~\eqref{eq:CPZ} when $r \leq 4$.

Combining Lemma~\ref{lm:CN-Turan} with~\eqref{eq:Erdos}, we also get the following Schwartz--Zippel type corollary, which may be of independent interest.

\begin{corollary}
Let $x_1^{d_1} \cdots x_r^{d_r}$ be a maximal monomial of $f$, where $d_1 \leq \dots \leq d_r$.
Then for any subsets $B_1, \dots, B_r$ of $\F$ with sizes $|B_i| = n$,
\[
| Z(f; B_1, \dots, B_r) | = O\left( n^{r - \frac{1}{(d_1 + 1) \cdots (d_{r - 1} + 1)}} \right).
\]
\end{corollary}

The described framework was also recently discussed in an article by Rote (see Section 8 in~\cite{roteGeneralizedCombinatorialLasonAlonZippelSchwartz2023}).

\section{Construction}\label{sec:Turan}

Here $\F_{p^r}$ is the finite field of size $p^r$ and $\F_{p^r}^* = \F_{p^r} \setminus \{0\}$.

\begin{lemma}\label{lm:construction}
Let $p$ be a prime number, and let $f \in \F_{p^r}[x_1, \dots, x_r]$ be the following polynomial:
\[
f(x_1, \dots, x_r) = x_1 \cdots x_r + \sum_{i = 1}^r \prod_{j = 1}^{r - 1} x_{i + j}^{p^r - p^j},
\]
where indices are interpreted modulo $n$, i.e. $x_{r + 1} = x_1$, $x_{r + 2} = x_2$, etc.
Then 
\[
|Z(f; \F_{p^r}^*, r)| = p^{r - 1} ( p^r - 1 )^{r - 1}.
\]
\end{lemma}
\begin{proof}
Note that for any $a_1, \dots, a_r \in \F_{p^r}^*$ we have $a_i^{p^r} = a_i$, so
\[
f(a_1, \dots, a_r) = a_1 \cdots a_r \left( 1 + \sum_{i = 1}^r \prod_{j = 0}^{r - 1} a_{i + j}^{-p^j} \right)
= a_1 \cdots a_r \left( 1 + \Tr \left( a_1^{-1} a_2^{-p} \cdots a_r^{-p^{r - 1}} \right) \right),
\]
where $\Tr(a) = a + a^p + \dots + a^{p^{r - 1}}$ is the trace of the field extension $\F_{p^r} / \F_p$. 

Now let us fix $a_2, \dots, a_r \in \F_{p^r}^*$.
As $a_1$ runs over all values of $\F_{p^r}^*$, so does $a_1^{-1} a_2^{-p} \cdots a_r^{-p^{r - 1}}$.
There are exactly $p^{r - 1}$ elements $a \in \F_{p^r}^*$ for which $\Tr(a) = -1$, i.e. for any fixed $a_2, \dots, a_r$ there are exactly $p^{r - 1}$ values of $a_1$ for which $f(a_1, \dots, a_r) = 0$.
\end{proof}

\begin{theorem}\label{thm:K2222}
For any $r \geq 2$,
\[
\ex(n, K^{(r)}_{2, \dots, 2}) = \Omega\left( n^{r - \frac{1}{r}} \right).
\]
\end{theorem}
\begin{proof}
Note that $x_1 \cdots x_r$ is a maximal monomial of the polynomial $f$ from Lemma~\ref{lm:construction}.
Thus, due to Lemma~\ref{lm:CN-Turan}, a hypergraph $H_p = H(f; \F_{p^r}^*, r)$ with $r(p^r - 1)$ vertices and $p^{r - 1} ( p^r - 1 )^{r - 1}$ edges is free of copies of $K^{(r)}_{2, \dots, 2}$ for every prime $p$, which gives the desired bound.
\end{proof}

\section{Concluding remarks}

The construction from Section~\ref{sec:Turan} in the case $r = 3$ is structurally similar to the one given by Katz, Krop and Maggioni in~\cite{katzRemarksBoxProblem2002}.
Their construction can be generalized to higher dimensions giving an alternative proof of Theorem~\ref{thm:K2222} (private communication with Cosmin Pohoata; see also Proposition 11.2 in~\cite{yangPropertiesShortestLength2021}).
Our approach gives a simpler construction and a much shorter proof.

Motivated by the ideas discussed in Section~\ref{sec:framework}, Rote posed a problem (see Problem 1 in~\cite{roteGeneralizedCombinatorialLasonAlonZippelSchwartz2023}), equivalent to asking how large can the set $Z(f; B_1, B_2)$ be for a polynomial of the form $f(x, y) = xy + P(x) + Q(y)$ and sets $B_1$, $B_2$ of size $n$ each.
Lemma~\ref{lm:construction} answers this question asymptotically if sets $B_1$ and $B_2$ are allowed to be taken from the finite field $\F_{p^2}$.

\section*{Acknowledgements}

I would like to thank Danila Cherkashin and Fedor Petrov for helpful discussions, and G{\"u}nter Rote for useful comments on a draft of this note.

\bibliographystyle{abbrv}
\bibliography{main}

\end{document}